\documentclass[preprint,12pt]{elsarticle}
\usepackage{amssymb}
\usepackage{pifont}
\usepackage{amsfonts}
\usepackage{amsmath}
\usepackage{graphicx}
\usepackage{latexsym,amsthm}
\usepackage{hyperref}
\usepackage{combelow}
\usepackage{setspace}
\usepackage[left=1.3cm,top=1.75cm,right=1.3cm]{geometry}

\makeatletter
\def\ps@pprintTitle{%
  \let\@oddhead\@empty
  \let\@evenhead\@empty
  \def\@oddfoot{\reset@font\hfil\thepage\hfil}
  \let\@evenfoot\@oddfoot
}
\makeatother
\allowdisplaybreaks

\newtheorem{lemma}{Lemma}

\newtheorem{theorem}{Theorem}
\numberwithin{equation}{section}
\begin{document}
\begin{frontmatter}
\title{On $(p,q)$ Baskakov-Durrmeyer-Stancu Operators}
\author[label1,label2,label*]{Vishnu Narayan Mishra}
\ead{vishnunarayanmishra@gmail.com,
vishnu\_narayanmishra@yahoo.co.in}
\address[label1]{Department of Applied Mathematics \& Humanities,
Sardar Vallabhbhai National Institute of Technology, Ichchhanath Mahadev Dumas Road, Surat -395 007 (Gujarat), India}
\address[label2]{L. 1627 Awadh Puri Colony, Phase –III, Beniganj , Opp. Industrial Training Institute(I.T.I.),
 Ayodhya Main Road, Faizabad-224 001, (Uttar Pradesh), India}
%\author[label3]{Narendra Kumar Govil}
%\ead{govilnk@auburn.edu}
%\address[label3]{Department of Mathematics and Statistics, Auburn University
%Auburn, AL 36849, USA}
\fntext[label*]{Corresponding author}
\author[label1]{Shikha Pandey}
\ead{sp1486@gmail.com}
\begin{abstract}
In the present paper, we introduce the generalized form of $(p,q)$ Baskakov-Durrmeyer Operators with Stancu type parameters. We derived the local and global approximation properties of these operators and obtained the convergence rate and behaviour for the Lipschitz functions. Moreover, we give comparisons and some illustrative graphics for the convergence of operators to some function.
\end{abstract}
\begin{keyword}
$(p,q)$-integers; $(p,q)$-Baskakov-Durrmeyer operators; Linear positive operator.\\
\textit{$2000$ Mathematics Subject Classification:} Primary $41A10$, $41A25$, $41A30$,$41A36$.
\end{keyword}
\end{frontmatter}
\section{Introduction}
In Post-quantum calculus, approximation by linear positive operators is an important research topic. $(p,q)$-calculus is used efficiently in various arena of science, for instance Lie group, differential equations, field theory, hypergeometric series etc (See \cite{Burban1, Burban2}). Various approximation operators has been defined which are based on $(p,q)$-integers. Some of these are Bernstein operators \cite{bernpq}, Bleimann-Butzer-Hahn operators \cite{bbhpq}, Sz\'{a}sz-Mirakyan operators \cite{szaszmirakyanpq} etc. (See \cite{V2,V3}).\\
For more detail and basic notations of $(p,q)$-calculus please refer \cite{baskaDurr,basicpq1,basicpq2,basicpq3}. Baskakov operators, its Durrmeyer type modification and its Stancu type genralizations has been well studied by various mathematicians (See \cite{G1,rnm1,rnm2,V1,W1}).
In order to approximate Lebesgue integrable functions Acar et al. defined $(p,q)$-Baskakov-Durrmeyer Operators \cite{baskaDurr} as
\begin{equation}\label{operator1}
\mathcal{B}_n^{p,q}(f;x)=[n]_{p,q}\sum_{k=0}^\infty b_{n,k}^{p,q}(x) \int_0^\infty p^{k(k-1)/2} \frac{([n]_{p,q})^k E_{p,q}(-q[n]_{p,q}t)}{[k]_{p,q}!} f\left(\frac{p^{k+n-1}}{q^{k-1}}t\right) d_{p,q}t,
\end{equation}
where $x\in [0,\infty), 0<q<p\leq 1$ and
\begin{equation*}
b_{n,k}^{p,q}(x)=\left[\begin{array}{c} n+k-1 \\ k \end{array} \right]_{p,q}p^{k+n(n-1)/2}q^{k(k-1)/2}\frac{x^k}{(1\oplus x)^{n+k}_{p,q}}, \quad E_{p,q}(x)=\sum_{n=0}^\infty \frac{q^{n(n-1)/2}x^n}{[n]_{p,q}!}
\end{equation*}
In this paper we generalize this operator \ref{operator1} with Stancu type parameters. Assuming $0\leq \alpha \leq \beta$, we define
\begin{equation}\label{operator2}
\mathcal{B}_{n,\alpha,\beta}^{p,q}(f;x)=[n]_{p,q}\sum_{k=0}^\infty b_{n,k}^{p,q}(x) \int_0^\infty p^{k(k-1)/2} \frac{([n]_{p,q})^k E_{p,q}(-q[n]_{p,q}t)}{[k]_{p,q}!} f\left(\frac{p^{k+n-1}q^{1-k}t[n]_{p,q}+\alpha}{[n]_{p,q}+\beta}\right) d_{p,q}t,
\end{equation}
for $x\in [0,\infty), 0<q<p\leq 1$.
\section{Moments}
\begin{lemma}\label{lemma1} For $x\in [0,\infty), 0<q<p\leq 1$, we have
\begin{enumerate}
\item[(i).] $\mathcal{B}_n^{p,q}(1;x)=1$,
\item[(ii).] $\mathcal{B}_n^{p,q}(t;x)=x+\frac{qp^{n-1}}{[n]_{p,q}}$
\item[(iii).]$\mathcal{B}_n^{p,q}(t^2;x)=\left(p^2+\frac{p^{n+2}}{q[n]_{p,q}}\right) x^2+\frac{p^{n-1}[2]_{p,q}^2}{[n]_{p,q}}x+\frac{[2]_{p,q}p^{2n-1}q^2}{[n]_{p,q}^2}.$
\end{enumerate}
\end{lemma}
\begin{lemma}\label{lemma2}For $x\in [0,\infty), 0<q<p\leq 1, 0\leq \alpha \leq \beta$, we have
\begin{enumerate}
\item[(i).] $\mathcal{B}_{n,\alpha,\beta}^{p,q}(1;x)=1$,
\item[(ii).] $\mathcal{B}_{n,\alpha,\beta}^{p,q}(t;x)=\frac{[n]_{p,q}}{([n]_{p,q}+\beta)}x+\frac{qp^{n-1}+\alpha}{[n]_{p,q}+\beta}$
\item[(iii).]$\mathcal{B}_{n,\alpha,\beta}^{p,q}(t^2;x)=\left(\frac{[n]_{p,q}^2 qp^2+[n]_{p,q} p^{n+2}}{q([n]_{p,q}+\beta)^2 }\right)x^2+\left(\frac{[n]_{p,q}[2]_{p,q}^2p^{n-1}+2\alpha [n]_{p,q}}{([n]_{p,q}+\beta)^2}\right)x+\frac{[2]_{p,q}q^2p^{2n-1}+2\alpha qp^{n-1}+\alpha^2}{([n]_{p,q}+\beta)^2}.$
\end{enumerate}
\end{lemma}
\begin{proof}
Using Lemma \ref{lemma1}, we can easily say, $\mathcal{B}_{n,\alpha,\beta}^{p,q}(1;x)=1$. Moreover
\begin{eqnarray*}
&&\mathcal{B}_{n,\alpha,\beta}^{p,q}(t;x)=[n]_{p,q}\sum_{k=0}^\infty b_{n,k}^{p,q}(x) \int_0^\infty p^{k(k-1)/2} \frac{([n]_{p,q})^k E_{p,q}(-q[n]_{p,q}t)}{[k]_{p,q}!} \left(\frac{p^{k+n-1}q^{1-k}t[n]_{p,q}+\alpha}{[n]_{p,q}+\beta}\right) d_{p,q}t\\&&
=\frac{[n]_{p,q}}{([n]_{p,q}+\beta)}\mathcal{B}_{n}^{p,q}(t;x)+\frac{\alpha}{[n]_{p,q}+\beta}\mathcal{B}_{n}^{p,q}(1;x)\\&&
=\frac{[n]_{p,q}}{([n]_{p,q}+\beta)}x+\frac{qp^{n-1}+\alpha}{[n]_{p,q}+\beta}.
\end{eqnarray*}and
\begin{eqnarray*}
&&\mathcal{B}_{n,\alpha,\beta}^{p,q}(t^2;x)=[n]_{p,q}\sum_{k=0}^\infty b_{n,k}^{p,q}(x) \int_0^\infty p^{k(k-1)/2} \frac{([n]_{p,q})^k E_{p,q}(-q[n]_{p,q}t)}{[k]_{p,q}!} \left(\frac{p^{k+n-1}q^{1-k}t[n]_{p,q}+\alpha}{[n]_{p,q}+\beta}\right)^2 d_{p,q}t\\&&
=\left(\frac{[n]_{p,q}}{([n]_{p,q}+\beta)}\right)^2\mathcal{B}_{n}^{p,q}(t^2;x)+\frac{2\alpha [n]_{p,q}}{([n]_{p,q}+\beta)^2}\mathcal{B}_{n}^{p,q}(t;x)+\frac{\alpha^2}{([n]_{p,q}+\beta)^2}\mathcal{B}_{n}^{p,q}(t;x)\\
&&=\left(\frac{[n]_{p,q}}{[n]_{p,q}+\beta}\right)^2\Biggl[\left(p^2+\frac{p^{n+2}}{q[n]_{p,q}}\right) x^2+\frac{p^{n-1}[2]_{p,q}^2}{[n]_{p,q}}x+\frac{[2]_{p,q}p^{2n-1}q^2}{[n]_{p,q}^2} \Biggl]+\frac{2\alpha [n]_{p,q}}{([n]_{p,q}+\beta)^2}\Biggl[ x+\frac{qp^{n-1}}{[n]_{p,q}}\Biggl]+\frac{\alpha^2}{([n]_{p,q}+\beta)^2}\\&&
=\left(\frac{[n]_{p,q}^2 qp^2+[n]_{p,q} p^{n+2}}{q([n]_{p,q}+\beta)^2 }\right)x^2+\Biggl(\frac{[n]_{p,q}[2]_{p,q}^2p^{n-1}+2\alpha [n]_{p,q}}{([n]_{p,q}+\beta)^2}\Biggl)x+\frac{[2]_{p,q}q^2p^{2n-1}+2\alpha qp^{n-1}+\alpha^2}{([n]_{p,q}+\beta)^2}.
\end{eqnarray*}
\end{proof}
Using Lemma\ref{lemma2} we can obtain
\begin{eqnarray*}
\mathcal{B}_{n,\alpha,\beta}^{p,q}((t-x);x)=\left(\frac{[n]_{p,q}}{([n]_{p,q}+\beta)}-1\right)x+\frac{qp^{n-1}+\alpha}{[n]_{p,q}+\beta},
\end{eqnarray*}
and
\begin{equation*}
\mathcal{B}_{n,\alpha,\beta}^{p,q}((t-x)^2;x)=\gamma_1(n)x^2+\gamma_2(n)x+\gamma_3(n),
\end{equation*}
where
\begin{equation*}
\gamma_1(n)=\left(\frac{[n]_{p,q}^2 qp^2+[n]_{p,q} p^{n+2}}{q([n]_{p,q}+\beta)^2 }\right)-\frac{2[n]_{p,q}}{([n]_{p,q}+\beta)}+1,
\end{equation*}
\begin{equation*}
\gamma_2(n)=\Biggl(\frac{[n]_{p,q}[2]_{p,q}^2p^{n-1}+2\alpha [n]_{p,q}}{([n]_{p,q}+\beta)^2}\Biggl)-\frac{2qp^{n-1}+\alpha}{[n]_{p,q}+\beta},
\end{equation*}
\begin{equation*}
\gamma_3(n)=\frac{[2]_{p,q}q^2p^{2n-1}+2\alpha qp^{n-1}+\alpha^2}{([n]_{p,q}+\beta)^2}.
\end{equation*}
Assuming $\gamma^*(n)=max\{\gamma_1(n),\frac{\gamma_2(n)}{2},\gamma_3(n)\}$, then
\begin{equation*}
\mathcal{B}_{n,\alpha,\beta}^{p,q}((t-x)^2;x)\leq \gamma^*(n)(1+x)^2.
\end{equation*}
\section{Local Approximation}
Let us consider the space of all bounded, real valued continuous function on $[0,\infty)$. The norm in this space is defined as
$$\|f\|_{C_B}=\sup_{x\in[0,\infty)}|f(x)|.$$
Consider the $\mathcal{K}$-functional
$$\mathcal{K}_2(f,\delta)=\inf_{g\in W^2}{\|f-g\|_{C_B}+\delta\|g^{''}\|_{C_B}},$$
where $\delta>0$ and $W^2=\{ g\in C_B[0,\infty):g^{'},g^{''}\in C_B[0,\infty)\}$.
 From \cite{devore}, (p. 177, Theorem 2.4), there exists an absolute constant $C>0$ such that
 \begin{equation}\label{def1}
 {K}_2(f,\delta)\leq \omega_2(f,\sqrt{\delta}),
 \end{equation}
where the second order modulus of smoothness of $f\in C_B[0,\infty)$ is $$\omega_2(f,\delta)=\sup_{0< h\leq \delta}\sup_{x\in [0,\infty)}|f(x+2h)-2f(x+h)+f(x)|,$$ and the usual modulus of continuity is $$\omega(f,\delta)=\sup_{0< h\leq \delta}\sup_{x\in [0,\infty)}|f(x+h)-f(x)|.$$
\begin{lemma}\label{lemma3}
For $f\in C_B[0,\infty), g\in C_B^2[0,\infty)$, we have
\begin{equation}
|\mathcal{\hat{B}}_{n,\alpha,\beta}^{p,q}(g;x)-g(x)|\leq \|g^{''}\|_{C_B}(\gamma^*(n)(1+x)^2+\mu_n^2(p,q,x)),
\end{equation}where
auxiliary operator \begin{equation} \label{operator3}
{\hat{B}}_{n,\alpha,\beta}^{p,q}(g;x)=\mathcal{B}_{n,\alpha,\beta}^{p,q}(g;x)+g(x)-g(\mathcal{B}_{n,\alpha,\beta}^{p,q}(t;x)).
\end{equation}
\end{lemma}
\begin{proof}By the definition of auxiliary operator, it can be concluded that
$\mathcal{\hat{B}}_{n,\alpha,\beta}^{p,q}(t-x;x)=0$.
Expanding the function $g$ using Taylor's expansion
$$g(t)=g(x)+g(x)(t-x)+\int_x^t(t-u)g^{''}(u)du.$$
Operating with \ref{operator3} on both sides of above equation, we get
\begin{eqnarray}\label{eq1}
&&|\mathcal{\hat{B}}_{n,\alpha,\beta}^{p,q}(g;x)-g(x)|=\mathcal{\hat{B}}_{n,\alpha,\beta}^{p,q}\left(\int_x^t(t-u)g^{''}(u)du;x\right)=\mathcal{B}_{n,\alpha,\beta}^{p,q}\left(\int_x^t(t-u)g^{''}(u)du;x\right) \nonumber \\&&\hspace{4 cm}-\left(\int^{\mathcal{B}_{n,\alpha,\beta}^{p,q}(t;x)}_x(\mathcal{B}_{n,\alpha,\beta}^{p,q}(t;x)-u)g^{''}(u)du\right)\nonumber \\
&&\hspace{2 cm}=\mathcal{B}_{n,\alpha,\beta}^{p,q}\left(\int_x^t(t-u)g^{''}(u)du;x\right)-\int^{\mathcal{B}_{n,\alpha,\beta}^{p,q}(t;x)}_x\left(\frac{[n]_{p,q}}{([n]_{p,q}+\beta)}x+\frac{qp^{n-1}+\alpha}{[n]_{p,q}+\beta}-u\right) g^{''}(u)du.\nonumber\\
\end{eqnarray}
On further investigation
\begin{equation}\label{eq2}
\left|\int_x^t(t-u)g^{''}(u)du\right|\leq \left|\int_x^t|t-u|~|g^{''}(u)|du\right|\leq \|g^{''}\|_{C_B}\left|\int_x^t|t-u|du\right|\leq \|g^{''}\|_{C_B} (t-x)^2,
\end{equation}and
\begin{eqnarray}\label{eq3}
\left|\int^{\mathcal{B}_{n,\alpha,\beta}^{p,q}(t;x)}_x\left(\frac{[n]_{p,q}}{([n]_{p,q}+\beta)}x+\frac{qp^{n-1}+\alpha}{[n]_{p,q}+\beta}-u\right) g^{''}(u)du\right|&\leq& \left(\frac{[n]_{p,q}}{([n]_{p,q}+\beta)}x+\frac{qp^{n-1}+\alpha}{[n]_{p,q}+\beta}-x\right)^2\|g^{''}\|_{C_B}\nonumber \\&=&\mu_n^2(p,q,x)\|g^{''}\|_{C_B}.
\end{eqnarray}
Substituting results of \ref{eq2} and \ref{eq3} into \ref{eq1} we get
$$|\mathcal{\hat{B}}_{n,\alpha,\beta}^{p,q}(g;x)-g(x)|\leq \|g^{''}\|_{C_B}(\gamma^*(n)(1+x)^2+\mu_n^2(p,q,x)).$$
\end{proof}
\begin{theorem}
For $f\in C_B[0,\infty)$ and $x\in [0,\infty)$, there exists a constant $L>0$ such that
$$|\mathcal{B}_{n,\alpha,\beta}^{p,q}(f;x)-f(x)|\leq L\omega_2(f;\sqrt{\gamma^*(n)(1+x)^2+\mu_n^2(p,q,x)})+\omega\left(f;\frac{qp^{n-1}+\alpha}{[n]_{p,q}+\beta}\right).$$
\end{theorem}
\begin{proof}
Using \ref{operator3} we can obtain for any $g\in C_B^2[0,\infty)$
\begin{eqnarray*}
&&|\mathcal{B}_{n,\alpha,\beta}^{p,q}(f;x)-f(x)|\leq |\mathcal{\hat{B}}_{n,\alpha,\beta}^{p,q}(f;x)-f(x)|+|f(x)-f(\mathcal{B}_{n,\alpha,\beta}^{p,q}(t;x))|
\\&&\hspace{3 cm} \leq |\mathcal{\hat{B}}_{n,\alpha,\beta}^{p,q}(f-g;x)-(f-g)(x)|+|f(x)-f(\mathcal{B}_{n,\alpha,\beta}^{p,q}(t;x))|+|\mathcal{\hat{B}}_{n,\alpha,\beta}^{p,q}(g;x)-g(x)|
\end{eqnarray*}
Using (\ref{operator2}),(\ref{operator3}) and lemma \ref{lemma2}, we have
$$\|\mathcal{\hat{B}}_{n,\alpha,\beta}^{p,q}(f;x)-f(x)\|\leq 4 \|f\|_{C_B}.$$
Now by Lemma \ref{lemma3}, we get
\begin{eqnarray*}
&&|\mathcal{B}_{n,\alpha,\beta}^{p,q}(f;x)-f(x)|\leq \|f-g\|_{C_B}+\left|f(x)-f\left(\frac{[n]_{p,q}}{([n]_{p,q}+\beta)}x+\frac{qp^{n-1}+\alpha}{[n]_{p,q}+\beta}\right)\right|\\&&\hspace{4 cm}+\|g^{''}\|_{C_B}(\gamma^*(n)(1+x)^2+\mu_n^2(p,q,x))
\end{eqnarray*}
Taking infimum on right hand side over all $g\in C_B^2[0,\infty)$ and using \ref{def1},
\begin{eqnarray*}
&&|\mathcal{B}_{n,\alpha,\beta}^{p,q}(f;x)-f(x)|\leq 4\mathcal{K}_2(f;(\gamma^*(n)(1+x)^2+\mu_n^2(p,q,x)))+\omega\left(f;\frac{qp^{n-1}+\alpha}{[n]_{p,q}+\beta}\right)
\\&&\hspace{3 cm}\leq 4C\omega_2 (f;\sqrt{\gamma^*(n)(1+x)^2+\mu_n^2(p,q,x)})+\omega\left(f;\frac{qp^{n-1}+\alpha}{[n]_{p,q}+\beta}\right).
\end{eqnarray*}
Taking $L=4C>0$, we get the desired result.
\end{proof}
\begin{theorem}
Let $f\in C_2[0,\infty)$, $p_n, q_n\in (0,1)$ such that $0<q_n<p_n\leq 1$ and $\omega_{a+1}(f,\delta)$ be its modulus of continuity on the finite interval $[0, a + 1] \subset [0,\infty)$, where $a>0$. Then, for every $n>3$,
$$|\mathcal{B}_{n,\alpha,\beta}^{p,q}(f;x)-f(x)|\leq 4M_f (1+a^2)(1+x)\sqrt{a^*(n)}+2\omega_{a+1}(f,(1+x)\sqrt{a^*(n)})$$ holds.
\end{theorem}
\begin{proof}
by \cite{ibikli}, $\omega_{a+1}(\cdot,\delta)$ has the property
$$|f(t)-f(x)|\leq 4 M_f(1+a^2)(t-x)^2+\left(1+\frac{|t-x|}{\delta}\right)\omega_{a+1}(f,\delta),~\delta>0$$
Applying Cauchy-Schwarz inequality and choosing $\delta=(1+x)\sqrt{a^*(n)}$, we have
\begin{eqnarray*}
&&|\mathcal{B}_{n,\alpha,\beta}^{p,q}(f;x)-f(x)|\leq 4 M_f(1+a^2)\mathcal{B}_{n,\alpha,\beta}^{p,q}((t-x)^2;x)+\omega_{a+1}(f,\delta)\left(1+\frac{1}{\delta}\sqrt{\mathcal{B}_{n,\alpha,\beta}^{p,q}((t-x)^2;x)}\right)\\&&
\hspace{2 cm}\leq 4 M_f(1+a^2)\gamma^*(n)(1+x)^2+\omega_{a+1}(f,\delta)\left(1+\frac{1}{\delta}\sqrt{\gamma^*(n)(1+x)^2}\right),
\end{eqnarray*}
which completes the proof.
\end{proof}
\section{Weighted Approximation}
Consider the weighted space functions defined as
\begin{enumerate}
\item[1.]$B_{x^2}[0,\infty)$ be the space of functions $f$ defined on $[0,\infty)$ satisfying $|f(x)|\leq M(1+x^2),M>0$.\\
\item[2.]$C_{x^2}[0,\infty)$ be the subspace of all continuous functions in $B_{x^2}[0,\infty)$.\\
\item[3.]$C^*_{x^2}[0,\infty)$ is the subspace of functions $f\in C_{x^2}[0,\infty)$ for which $\lim_{x\to \infty}\frac{f(x)}{1+x^2}$ is finite.
\end{enumerate}
Note that the space $B_{x^2}[0,\infty)$ is a normed linear space with the norm $\|f\|_{x^2}=\sup_{x\geq 0}\frac{|f(x)|}{1+x^2}$. In order to calculate rate of convergence consider the weighted modulus of continuity defined as
$$\Omega(f;\delta)=\sup_{x\geq 0, 0<h\leq \delta}\frac{|f(x+h)-f(x)|}{1+(x+h)^2} ~~~for~~f\in C^*_{x^2}[0,\infty). $$
\begin{lemma}\label{lemma4}
Weighted modulus of continuity has following properties as defined in \cite{lopez},
\begin{enumerate}
\item[1.] $\Omega(f;\delta)$ is a monotonic increasing function of $\delta$.
\item[2.] $\lim_{\delta\to 0^+}\Omega(f;\delta)=0.$
\item[3.] For any $\lambda \geq 0, ~\Omega(f;\lambda \delta)\leq (1+\lambda) \Omega(f;\delta).$
\end{enumerate}
\end{lemma}
\begin{theorem}
Consider $p=p_n$ and $q=q_n$ such that $0<q_n<p_n\leq 1$ and as $n\to \infty$, we get $p_n\to1, q_n\to1, p_n^n\to 1, q_n^n \to 1$. For any $f\in C_{x^2}^*[0,\infty)$, we have
$$\lim_{n\to \infty}\sup_{x \in [0,\infty)}\frac{\left|\mathcal{B}_{n,\alpha,\beta}^{p_n,q_n}(f;x)-f(x)\right|}{(1+x^2)^{1+\alpha}}=0.$$
\end{theorem}
\begin{proof}
\begin{eqnarray}\label{d1}
&&\sup_{x \in [0,\infty)}\frac{\left|\mathcal{B}_{n,\alpha,\beta}^{p_n,q_n}(f;x)-f(x)\right|}{(1+x^2)^{1+\alpha}} \leq \sup_{x\leq x_0}\frac{\left|\mathcal{B}_{n,\alpha,\beta}^{p_n,q_n}(f;x)-f(x)\right|}{(1+x^2)^{1+\alpha}}+\sup_{x \geq x_0}\frac{\left|\mathcal{B}_{n,\alpha,\beta}^{p_n,q_n}(f;x)-f(x)\right|}{(1+x^2)^{1+\alpha}}\nonumber
\\&& \hspace{3 cm}\leq  \|\mathcal{B}_{n,\alpha,\beta}^{p_n,q_n}(f)-f\|_{C[0,x_0]}+\|f\|_{x^2}\sup_{x \geq x_0}\frac{\left|\mathcal{B}_{n,\alpha,\beta}^{p_n,q_n}(1+t^2;x)\right|}{(1+x^2)^{1+\alpha}}+\sup_{x \geq x_0}\frac{|f(x)|}{(1+x^2)^{1+\alpha}}.
\end{eqnarray}
where $\|\cdot\|_{C[0,x_0]}$ is the uniform norm on $[0,x_0]$.\\ We have $\sup_{x\geq x_0}\frac{|f(x)|}{(1+x^2)^{1+\alpha}}\leq \frac{\|f\|_{x^2}}{(1+x^2_0)^{1+\alpha}} $ since $|f(x)|\leq M(1+x^2)$. For arbitrary $\epsilon$ we can take large value of $x_0$  such that
\begin{equation}\label{eq4}
\frac{\|f\|_{x^2}}{(1+x^2_0)^{1+\alpha}}< \frac{\epsilon}{3}.
\end{equation}
Using lemma \ref{lemma2} we get
\begin{equation}\label{eq5}
\|f\|_{x^2}\frac{\left|\mathcal{B}_{n,\alpha,\beta}^{p_n,q_n}(1+t^2;x) \right|}{(1+x^2)^{1+\alpha}}=\frac{(1+x^2)}{(1+x^2)^{1+\alpha}}\|f\|_{x^2}\leq \frac{\|f\|_{x^2}}{(1+x^2)^{\alpha}}\leq \frac{\|f\|_{x^2}}{(1+x^2_0)^{\alpha}}<\frac{\epsilon}{3}.
\end{equation}
Also by Korovkin's theorem we have
\begin{equation}\label{eq6}
\|\mathcal{B}_{n,\alpha,\beta}^{p_n,q_n}(f)-f\|_{C[0,x_0]}\leq \frac{\epsilon}{3}.
\end{equation}
Hence, using (\ref{eq4})-(\ref{eq6}) in the inequality (\ref{d1}), we get the desired result.
\end{proof}
\begin{theorem}
If $f\in C^*_{x^2}[0,\infty)$, then as $n\to \infty$ we have
\begin{eqnarray*}
|\mathcal{B}_{n,\alpha,\beta}^{p_n,q_n}(f;x)-f(x)|&\leq & (1+x^2)(\nu_2+\nu_4\sqrt{x^2+x+1})~\Omega(f;\delta)
\\&\leq& \nu (1+x^{2+\lambda})~\Omega(f;\delta),
\end{eqnarray*}
where $\lambda\geq 1, \delta_n=max{\alpha_n,\beta_n,\gamma_n}$ and $\nu$ is a positive constant independent of $f$ and $n$.
\end{theorem}
\begin{proof}
By the definition of weighted modulus of continuity and Lemma \ref{lemma4}, we get
\begin{eqnarray*}
|f(t)-f(x)|&\leq & \left(1+(x+|t-x|)^2\right)\left(1+\frac{|t-x|}{\delta}\right)~\Omega(f;\delta)
\\&\leq & \left(1+(2x+t)^2\right)\left(1+\frac{|t-x|}{\delta}\right)~\Omega(f;\delta).
\end{eqnarray*}
Operating with \ref{operator2} both sides of inequality and using Cauchy-Schwarz inequality we get
\begin{eqnarray*}
|\mathcal{B}_{n,\alpha,\beta}^{p_n,q_n}(f;x)-f(x)|&\leq &\left(\mathcal{B}_{n,\alpha,\beta}^{p_n,q_n}\left(1+(2x+t)^2;x\right)+\mathcal{B}_{n,\alpha,\beta}^{p_n,q_n}\left(\left(1+(2x+t)^2\right)\frac{|t-x|}{\delta};x\right)\right)~\Omega(f;\delta)
\\&\leq & \Biggl(\mathcal{B}_{n,\alpha,\beta}^{p_n,q_n}\left(1+(2x+t)^2;x\right)+\frac{1}{\delta_n}\sqrt{\mathcal{B}_{n,\alpha,\beta}^{p_n,q_n}\left(\left(1+(2x+t)^2\right)^2;x\right)}\\&& \hspace{.5 cm}\times \sqrt{\mathcal{B}_{n,\alpha,\beta}^{p_n,q_n}\left((t-x)^2;x\right)}\Biggl)~\Omega(f;\delta).
\end{eqnarray*}
For large values of $n$, from Lemma \ref{lemma2}, we can obtain
\begin{equation}\label{eq7}
\frac{\mathcal{B}_{n,\alpha,\beta}^{p_n,q_n}(1+t^2;x)}{1+x^2}\leq 1+\nu_1,
\end{equation}
where $\nu_1$ is a positive constant. From \ref{eq7} we can get
\begin{equation}
\mathcal{B}_{n,\alpha,\beta}^{p_n,q_n}(1+(2x+t)^2;x)\leq \nu_2 (1+x^2),~~\nu_2>0,
\end{equation}
In the same way we can get, $\frac{\mathcal{B}_{n,\alpha,\beta}^{p_n,q_n}(1+t^4;x)}{1+x^4}\leq 1+\nu_3,~ \nu_3>0$ and hence for large values of $n$, $$\sqrt{\mathcal{B}_{n,\alpha,\beta}^{p_n,q_n}((1+(2x+t)^2)^2;x)}\leq \nu_4 (1+x^2),~~\nu_4>0.$$
Hence, we have
\begin{equation*}
|\mathcal{B}_{n,\alpha,\beta}^{p_n,q_n}(f;x)-f(x)| \leq (1+x^2)\left(\nu_2+\frac{1}{\delta_n}\nu_4 (1+x)\sqrt{\gamma^*(n)}\right)~\Omega(f;\delta).
\end{equation*}
Let $\delta_n=\sqrt{\gamma^*(n)}$
\begin{eqnarray*}
&&|\mathcal{B}_{n,\alpha,\beta}^{p_n,q_n}(f;x)-f(x)| \leq (1+x^2)(\nu_2+\nu_4 (1+x))~\Omega(f;\sqrt{\gamma^*(n)})
\\&& \hspace{3 cm} \leq \nu (1+x^{2+\lambda})~\Omega(f;\sqrt{\gamma^*(n)}).
\end{eqnarray*}
where $\nu:=\nu_2+\nu_4.$
\end{proof}
\section{Graphics}
We present comparisons and illustrative graphics for the convergence of $(p,q)$ Baskakov-Durrmeyer-Stancu Operators to the function $f(x)=\cos x^2$, for different values of parameters $p=0.9,q=0.8,n=100,98,\alpha=0.1,\beta=0.5$.
\begin{figure}[h!]
\begin{center}
\includegraphics[scale=0.8]{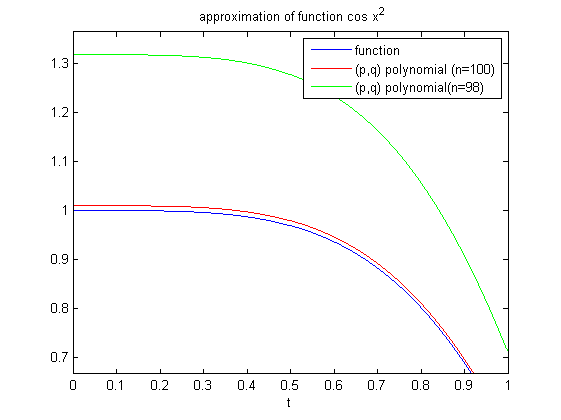}
\end{center}
\end{figure}
\begin{flushleft}
\textbf{Conflict of Interest:} The authors declare that they have no competing interests regarding the publication of this manuscript.
\end{flushleft}
\begin{flushleft}
\textbf{References}
\end{flushleft}


\begin{thebibliography}{99}
\bibitem{szaszmirakyanpq}T. Acar, (p, q)-generalization of Sz´asz-Mirakyan operators, Math. Method. Appl. Sci., DOI:
10.1002/mma.3721.
\bibitem{baskaDurr}T. Acar, A. Aral and M. Mursaleen, Approximation by Baskakov-Durrmeyer operators based on (p,q) -integers,
arXiv:submit/1450876 [math.CA]
\bibitem{Burban1}I. M. Burban, Two-parameter deformation of the oscillator albegra and (p, q) analog of two
dimensional conformal field theory, Nonlinear Math. Phys., 2 (3-4), (1995), 384-391.
\bibitem{Burban2}I. M. Burban, A. U. Klimyk, P,Q differentiation, P,Q integration and P,Q hypergeometric
functions related to quantum groups, Inter. Trans. Spec. Fucnt., 2 (1), (1994), 15-36.
\bibitem{devore}R. A. Devore, G. G. Lorentz, Constructive Approximation, Springer, Berlin, 1993.
\bibitem{G1} A.R. Gairola, Deepmala, L.N. Mishra, Rate of Approximation by Finite Iterates of q-Durrmeyer Operators, Proceedings of the National Academy of Sciences, India Section A: Physical Sciences, (2016), doi: 10.1007/s40010-016-0267-z, in press.
\bibitem{basicpq1}M. N. Hounkonnou, J. D´esir´e, B. Kyemba, R(p, q)-calculus: differentiation and integration,
SUT Journal of Mathematics, Vol. 49, No. 2 (2013), 145-167.
\bibitem{rnm1}G. I\c{c}\"{o}z, Ram N. Mohapatra, Approximation properties by $q$-Durrmeyer-Stancu operators. Anal. Theory Appl. 29 (2013), no. 4, 373-383.
\bibitem{rnm2}G. I\c{c}\"{o}z, R.N. Mohapatra, Weighted approximation properties of Stancu type modification of q- Szász-Durrmeyer operators, Commun. Fac. Sci. Univ. Ank. Sér. A1 Math. Stat., Volume 65, Number 1, (2016), pp. 87-103.
\bibitem{ibikli}E. Ibikli, E.A. Gadjieva, The order of approximation of some unbounded function by the
sequences of positive linear operators, Turk. J. Math., 19(3), 331-337 (1995).
\bibitem{basicpq2}R. Jagannathan, K. S. Rao, Two-parameter quantum algebras, twin-basic numbers, and
associated generalized hypergeometric series, Proceedings of the International Conference on
Number Theory and Mathematical Physics, 20-21 December 2005.
\bibitem{lopez}AJ, L\'{o}pez-Moreno, Weighted simultaneous approximation with Baskakov type operators,
Acta Math. Hung. 104 (1-2), 143-151 (2004).
\bibitem{V1}  V.N. Mishra, K. Khatri, L.N. Mishra, Deepmala, Inverse result in simultaneous approximation by Baskakov-Durrmeyer-Stancu operators, Journal of Inequalities and Applications 2013, 2013:586. doi:10.1186/1029-242X-2013-586.
\bibitem{V2} V.N. Mishra, M. Mursaleen, S. Pandey, Approximation properties of Chlodowsky variant of $(p, q)$
Bernstein-Stancu-Schurer operators, arXiv:1510.00405, 2015.
\bibitem{V3} V.N. Mishra, S. Pandey, On Chlodowsky variant of $(p, q)$ Kantorovich-Stancu-Schurer operators , arXiv:1508.06888, 2015.
\bibitem{bernpq}M. Mursaleen, K. J. Ansari and A. Khan, On (p, q)-analogue of Bernstein Operators, Appl.
Math. Comput., 266(2015) 874-882.
\bibitem{bbhpq}M. Mursaleen, Md. Nasiruzzaman, Asif Khan and K. J. Ansari, Some approximation results
on Bleimann-Butzer-Hahn operators defined by (p, q)-integers, Filomat (accepted).
\bibitem{basicpq3}V. Sahai and S. Yadav, Representations of two parameter quantum algebras and p, q-special
functions, J. Math. Anal. Appl. 335(2007), 268-279.

\bibitem{W1} A. Wafi, N. Rao, Deepmala, Approximation properties by generalized-Baskakov-Kantorovich-Stancu type operators, Appl. Math. Inf. Sci. Lett., Vol.  4, No. 3, (2016), pp. 1-8.
\end{thebibliography}
\end{document}